\documentclass[11pt]{amsart}
\usepackage{amsthm}
\usepackage{array}
\usepackage{amsmath}
\usepackage{enumerate}
\usepackage{tikz}
\usetikzlibrary{calc}
\usepackage{amssymb}
\usepackage{latexsym}
\usepackage{amsfonts}
\usepackage{color}
\usepackage{mathrsfs}
\usepackage{epsfig,helvet}

\theoremstyle{plain} \numberwithin{equation}{section}
\newtheorem{thm}{Theorem}[section]
\newtheorem{theorem}[thm]{Theorem}
\newtheorem{lemma}[thm]{Lemma}
\newtheorem{corollary}[thm]{Corollary}

\newtheorem{definition}[thm]{Definition}
\newtheorem{proposition}[thm]{Proposition}
\begin{document}
\setcounter{page}{1}

\title[Isoclinic relative central extensions of a pair of regular Hom-Lie algebras]{Isoclinic relative central extensions of a pair of regular Hom-Lie algebras}

\author[Padhan]{Rudra Narayan Padhan}
\address{Institute of Technical Education and Research  \\
		Siksha `O' Anusandhan (A Deemed to be University)\\
		Bhubaneswar-751030 \\
		Odisha, India }
\email{ rudra.padhan6@gmail.com, rudranarayanpadhan@soa.ac.in}

\author[Khuntia]{Tofan Kumar Khuntia}
\address{Department of Mathematics, National Institute of Technology  \\
         Rourkela, 
          Odisha-769008 \\
                India}
\email{tofan123khuntia@gmail.com}

\subjclass[2020]{Primary 17B61.}
\keywords{ Hom-Lie algebra; relative central extension; isoclinism }
\maketitle

\begin{abstract}
In this paper, we show the relation among the relative central extensions in an isoclinism family of a particular relative central extension of Hom-Lie algebras. We define the notion of isoclinism on the central relative extensions of a pair of Hom-Lie algebras. Then, we figure out the concept of isomorphism in the equivalence class of isoclinisms on the central relative extensions of a pair of Hom-Lie algebras.
\end{abstract}

\section{\textbf{Introduction and preliminaries}}
The concept of Hom-Lie algebras was first introduced by Hartwig et al. \cite{h2006}, as a part of a study of deformations of the Witt and the Virasoro algebras. A Hom-Lie algebra is basically a non-associative algebra satisfying the skew-symmetry and the Hom-Jacobi identity twisted by a linear map. If we take this linear map as the identity map, then the Hom-Lie algebra becomes a Lie algebra. Therefore, the generalizations of the known theories from Lie algebras to the Hom-Lie play a crucial role as these generalizations are not straight forward. Also, because of the close relation to discrete and deformed vector fields and differential calculus, Hom-Lie algebras are widely studied recently \cite{c2015, c2017, p2020}.

The notion of isoclinism for Lie algebras was first introduced by Moneyhun \cite{km1994} in 1994. Eshrati et al. \cite{es2016} studied various properties of isoclinism in $n$-Lie algebras. Isoclinism of various algebraic structures have been exploid, for more details see \cite{kh2022,m2009, m2013, m2018, sn2020, p2020, rud2022, r2020, sa2021, saf2021,n2022}. Recently in 2020, Arabyani and Sheikh-mohseni  \cite{ar2020} introduced the concept of isoclinism on the relative central extensions of pairs of Lie algebras. Lately, isoclinism has been defined and studied for Hom-Lie algebra by Padhan et al. \cite{p2020}. In this paper, we study the relative central extensions of pairs of Hom-Lie algebras.
%Recently in 2022, Khuntia et al. \cite{kh2022} extended this concept to the case of $n$-isoclinism in Lie superalgebras. In 2020, Arabyani and Sheikh-mohseni in \cite{ar2020} introduced the concept of isoclinism on the relative central extensions of pairs of Lie algebras. %Here we have basically extended the above work done by Arabyani and Sheikh-mohseni to the case of Hom-Lie algebras. 
%In \cite{m2009, m2013, m2018, sn2020, p2020, r2020, sa2021, saf2021}, we can also find similar kind of works.

Firstly,x we define the notions such as isoclinism between a pair of Hom-Lie algebras, Hom-action and relative central extensions. In Section 2, we introduce the concept of isoclinism on the relative central extensions of a pair of Hom-Lie algebras which is a generalization of the work in \cite{ar2020}. We also give some equivalent conditions under which two relative central extensions are isoclinic. In Section 3, we define the concept of factor sets in the context of relative central extensions of a pair of Hom-Lie algebras. Then we show that under certain conditions, two concepts of isoclinism and isomorphism between the relative central extensions of a pair of finite dimensional Hom-Lie algebras are same. In fact, we show that under some conditions two finite dimensional relative central extensions are isoclinic if and only if they are isomorphic.

\begin{definition}
A Hom-Lie algebra $(L,\alpha_L)$ is a non-associative algebra $L$ together with a bilinear map $[.,.]:L\times L\rightarrow L$ and a linear map $\alpha_L:L\rightarrow L$ satisfying 
\vspace{0.2cm} 

\hspace{1cm} $[x,y]=-[y,x]$ \hfill (skew-symmetry)
\vspace{0.1cm} 

\hspace{1cm} $[\alpha_L(x),[y,z]]+[\alpha_L(y),[z,x]]+[\alpha_L(z),[x,y]]=0$ \hfill (Hom-Jacobi identity)
\vspace{0.2cm} 

for all $x,y,z\in L$.
\end{definition}

A Hom-Lie subalgebra $(H,\alpha_H)$ of a Hom-Lie algebra $(L,\alpha_L)$ is a vector subspace $H$ of $L$ closed under the product, i.e., $[x,y]\in H~\forall ~x,y\in H$, together with the endomorphism $\alpha_H:H\rightarrow H$ being the restriction of $\alpha_L$ on $H$. A Hom-Lie subalgebra $(H,\alpha_H)$ of $(L,\alpha_L)$ is said to be an ideal if $[x,y]\in H~\forall ~x\in H,~y\in L$. The center of a Hom-Lie algebra $(L,\alpha_L)$ is the vector subspace $$Z(L)=\{x\in L~|~[x,y]=0~\forall~y\in L\}.$$
A Hom-Lie algebra $(L,\alpha_L)$ is said to be $multiplicative$ if  $\varphi([l_1,l_2])=[\alpha(l_1),\alpha(l_2)]$ for all $l_1,l_2\in L$. A multiplicative Hom-Lie algebra $(L,\alpha_L)$ is said to be $regular$ if $\alpha$ is bijective.

\begin{definition}
Let $(L_1,\alpha_1)$ and $(L_2,\alpha_2)$ be two Hom-Lie algebras. Then a morphism from $L_1$ to $L_2$ is a linear map $f:L_1\rightarrow L_2$ such that $f([l,m]_{L_1})=[f(l),f(m)]_{L_2}$ for all $l,m\in L_1$ and $f\circ\alpha_{1}=\alpha_{2}\circ f $ i.e., the following diagram commutes;
\vspace{0.0001cm}

 \begin{center}
  \begin{tikzpicture}[>=latex]
\node (x) at (0,0) {\(L_1\)};
\node (z) at (0,-2) {\(L_1\)};
\node (y) at (2.5,0) {\(L_2\)};
\node (w) at (2.5,-2) {\(L_2\)};
\draw[->] (x) -- (y) node[midway,above] {f};
\draw[->] (x) -- (z) node[midway,left] {$\alpha_1$};
\draw[->] (z) -- (w) node[midway,above] {f};
\draw[->] (y) -- (w) node[midway,right] {$\alpha_2$};
\end{tikzpicture}\\
 \end{center} 
\end{definition}

\begin{definition}
Let $(M,\alpha_M)$ be an ideal of a Hom-Lie algebra $(L,\alpha_L)$. Then $(M,L)$ is said to be a pair of Hom-Lie algebras and we may define the commutator as 
$$[M,L]=\langle [m,l]:m\in M,~l\in L\rangle$$
and the centre of $(M,L)$ can be defined as  
 $$Z(M,L)=\{m\in M:[m,l]=0, \forall~l\in L\}.$$
\end{definition}

If $L$ is regular Hom-Lie algebra, then $[M,L]$ and $Z(M,L)$ are ideals of $L$. Now onwards we always assume that the the given Hom-Lie algebra is regular.

\begin{definition}
Let $(M_1,L_1)$ and $(M_2,L_2)$ be two pairs of Hom-Lie algebras. These two pairs are said to be \textit{isoclinic}, denoted by $(M_1,L_1)\sim(M_2,L_2)$, if there exists a pair of Hom-Lie isomorphisms $(\varphi,\theta)$ where, $\varphi:L_1/Z(M_1,L_1)\rightarrow L_2/Z(M_2,L_2)$ is such that $\varphi(\bar{M_1})=\bar{M_2}$, $\bar{M_j}=M_j/Z(M_j,L_j),~j=1,2$ and $\theta:[M_1,L_1]\rightarrow [M_2,L_2]$ satisfies $\theta([m_1,l_1])=[m_2,l_2],$ for all $m_1\in M_1, l_1\in L_1, \bar{m_2}\in \varphi(\bar{m_1}), \bar{l_2}\in \varphi(\bar{l_1})$, and the following diagram is commutative:
\vspace{0.001cm} \\
 \begin{center}
  \begin{tikzpicture}[>=latex]
\node (x) at (0,0) {\(M_1/Z(M_1,L_1)\times L_1/Z(M_1,L_1) \)};
\node (z) at (0,-3) {\(M_2/Z(M_2,L_2)\times L_2/Z(M_2,L_2)\)};
\node (y) at (6,0) {\([M_1,L_1]\)};
\node (w) at (6,-3) {\([M_2,L_2]\)};
\draw[->] (x) -- (y) node[midway,above] {$\mu$};
\draw[->] (x) -- (z) node[midway,right] {$\varphi^{2}$};
\draw[->] (z) -- (w) node[midway,above] {$\rho$};
\draw[->] (y) -- (w) node[midway,right] {$\theta$};
\end{tikzpicture}\\
 \end{center} 
 \end{definition}

Observe that when $M_i = L_i$ for $i=1,2$ on the above definition, we have the definition for isoclinism  of Hom-Lie algebras \cite{p2020}.

\begin{definition}
Let $(L,\alpha_L)$ and $(K,\alpha_K)$ be two Hom-Lie algebras. A Hom-action of $L$ on $K$ is a linear map $L\otimes K\rightarrow K,~x\otimes k\mapsto ~^xk$, satisfying the following properties:
\vspace{0.1cm}
\begin{enumerate}
\item $^{[x,y]}\alpha_K(k)=~^{\alpha_L{(x)}}(^yk)-~^{\alpha_L{(y)}}(^xk)$
\vspace{0.1cm}
\item $^{\alpha_L{(x)}}[k,k^\prime]=[^xk,\alpha_K(k^\prime)]+[\alpha_K(k),~^x(k^\prime)]$ 
\vspace{0.1cm}
\item $\alpha_K(^xk)=~^{\alpha_L{(x)}}\alpha_K(k)$ 
\end{enumerate}
\vspace{0.1cm}
for all $x,y\in L,$ and $k,k^\prime \in K$.
\end{definition}

It is easy to see that if $K$ is an ideal of $L$, then by Lie multiplication we have an action of $L$ on $K$.

\begin{definition}
A relative central extension of a pair of Hom-Lie algebras $(M,L)$ is a homomorphism of Hom-Lie algebras $\sigma:(M^*,\alpha_*)\rightarrow (L,\alpha_L)$ together with an action of $L$ on $M^*$ satisfying the following conditions:
\begin{enumerate}
\item $\sigma(M^*)=M,$
\item $\sigma(^l{\alpha_*(m)})=[\alpha_L(l),\sigma(\alpha_*(m))],~$ for all $l\in L,~m\in M^*$,
\item $^{\sigma(m^\prime)}\alpha_*(m)=[m^\prime,m],~$ for all $m,m^\prime\in M^*$,
\end{enumerate}
\end{definition}

For a given relative central extension $\sigma: M^*\rightarrow L$, the $L$-commutator and the $L$-central subalgebras of $M^*$ are defined respectively, as follows: $$[M^*,L]=\{~^l\alpha_*(m)|~l\in L,m\in M^*\}$$ and $$Z(M^*,L)=\{m\in M^*|~^l\alpha_*(m)=0,\forall~ l\in L\}.$$ 
Observe that $\textnormal{Ker}~\sigma\subseteq Z(M^*,L)$ and $\textnormal{Ker}~\sigma\subseteq Z(M^*)$.

\vspace{0.3cm} 

\section{\textbf{Isoclinic relative extensions}}

In this section, we introduce the concept of isoclinism on the relative central extensions of pairs of Hom-Lie algebras.

\begin{definition}
Let $(M_1,L_1)$ and $(M_2,L_2)$ be two pairs of Hom-Lie algebras. Let there be two Hom-Lie homomorphisms $\gamma:(L_1,\alpha_1)\rightarrow(L_2,\alpha_2)$ and $\beta: (M_1^*,\alpha_1^*)\rightarrow (M_2^*,\alpha_2^*)$ with $\gamma(M_1)=M_2$ and $\beta([M_1^*,L_1])=[M_2^*,L_2]$ such that $\gamma \sigma_1(m_1)=\sigma_2 \beta(m_1)$, where $\sigma_i:(M_i^*,\alpha_i^*)\rightarrow (L_i,\alpha_i),~(i=1,2)$ are relative central extensions of $(M_i,L_i),~(i=1,2)$ respectively. Then the pair $(\gamma,\beta):\sigma_1\rightarrow \sigma_2$ is called a morphism from $\sigma_1$ to $\sigma_2$. If $\gamma$ and $\beta$ are isomorphic then the pair $(\gamma,\beta)$ is called an isomorphism. The isomorphism between $\sigma_1$ and $\sigma_2$ is denoted by $\sigma_1\cong\sigma_2$.
\end{definition}

\begin{definition}
The relative central extensions $\sigma_1:(M_1^*,\alpha_1^*)\rightarrow (L_1,\alpha_1)$ and $\sigma_2:(M_2^*,\alpha_2^*)\rightarrow (L_2,\alpha_2)$ are said to be isoclinic if there exist Hom-Lie isomorphisms $\gamma:(L_1,\alpha_1)\rightarrow(L_2,\alpha_2)$ with $\gamma(M_1)=M_2$ and $\beta^\prime: [M_1^*,L_1]\rightarrow [M_2^*,L_2]$ such that for all $m_1\in M_1^*,l_1\in L_1$ we have $\beta^\prime(^{l_1}\alpha_1^*(m_1))=~ ^{l_2}\alpha_2^*(m_2)$, where $m_2\in M_2^*,l_2\in L_2$ and  $\gamma\sigma_1(\alpha_1^*(m_1))=\sigma_2(\alpha_2^*(m_2)),\gamma\alpha_1(l_1)=\alpha_2(l_2).$ In this case, the pair $(\gamma,\beta^\prime)$ is called an isoclinism from $\sigma_1$ to $\sigma_2$ and we write $(\gamma,\beta^\prime):\sigma_1\cong\sigma_2$. A morphism $(\gamma, \beta):\sigma_1 \rightarrow \sigma_2$ is called isoclinic if the pair $(\gamma,\beta|_{[M_1^*,L_1]})$ is an isoclinism from $\sigma_1$ to $\sigma_2$.
\end{definition}

Suppose that $(M_1,L_1)$ and $(M_2,L_2)$ be two pairs of Hom-Lie algebras and let for $i=1,2,~\overline{M_i}=M_i/Z(M_i,L_i),~\overline{L_i}=L_i/Z(M_i,L_i)$. If $(\gamma,\beta^\prime)$ be an isoclinism between relative central extensions $\overline{\sigma_1}:\overline{M_1}\rightarrow\overline{L_1}$ and $\overline{\sigma_2}:\overline{M_2}\rightarrow\overline{L_2}$, then $\beta^\prime([\overline{m_1},\overline{l_1}])=[\overline{m_2},\overline{l_2}]$, where $\gamma(\overline{l_1})=\overline{l_2}$ and $\gamma(\overline{m_1})=\overline{m_2}$. Now, if we consider $\beta:[M_1,L_1]\rightarrow [M_2,L_2]$, a Hom-Lie isomorphism defined by $\beta([m_1,l_1])=[m_2,l_2]$, then the pair $(\gamma, \beta)$ will be an isoclinism from $(M_1,L_1)$ to $(M_2,L_2)$. Therefore, two pairs of Hom-Lie algebras $(M_1,L_1)$ and $(M_2,L_2)$ are isoclinic if the relative central extensions $\overline{\sigma_1}$ and $\overline{\sigma_2}$ are isoclinic.

Now we give some examples of isoclinic morphisms.

\vspace{0.1cm} 

\noindent{\sc Example 1.} Let $(L,\alpha_L)$ be a Hom-Lie algebra and $A$ be an abelian Hom-Lie subalgebra of $L$, then the maps $\pi_L:L\times A\rightarrow L$ and $i_L:L\rightarrow L\times A$ denote the projective and canonical Hom-Lie homomorphisms, respectively. Let $\sigma:(M^*,\alpha_*)\rightarrow (L,\alpha_L)$ be a relative central extension of the pair $(M,L)$. Then the morphism $\sigma\pi_{M^*}:M^*\times A\rightarrow L$ is a relative central extension of $(M,L)$ together with an action of $L$ on $M^*\times A$ defined by $^l(m,a)=(^lm,a)$. It can be easily verified that $(1_L,\pi_{M^*}):\sigma\pi_{M^*}\rightarrow \sigma$ and $(1_L,i_{M^*}):\sigma\rightarrow \sigma\pi_{M^*}$ are isoclinic epimorphism and isoclinic monomorphism, respectively.

\vspace{0.3cm} 

\noindent{\sc Example 2.} Let $\sigma:(M^*,\alpha_*)\rightarrow (L,\alpha_L)$ be a relative central extension of the pair $(M,L)$ and $N$ be an ideal of $M^*$ such that $N\subseteq \textnormal{Ker}{\sigma}$. Let $\bar{\sigma}:M^*/N\rightarrow L$ be the morphism induced by $\sigma$. This $\bar{\sigma}$ along with an action of $L$ on $M^*/N$ given by $^l(m+N)=(~^lm+N)$ forms a relative central extension of $(M,L)$. Let $\rho_{M^*}:M^*\rightarrow M^*/N$ denote the natural epimorphism. Then we find an isoclinic epimorphism $(1_L,\rho_{M^*}):\sigma\rightarrow\bar{\sigma}$ from $\sigma$ to $\bar{\sigma}$.

\begin{lemma} \label{lem1}
Let $\sigma_1:(M_1^*,\alpha_1^*)\rightarrow (L_1,\alpha_1)$ and $\sigma_2:(M_2^*,\alpha_2^*)\rightarrow (L_2,\alpha_2)$ be relative central extensions of the pairs of Hom-Lie algebras $(M_1,L_1)$ and $(M_2,L_2)$, respectively. If $(\gamma,\beta^\prime):\sigma_1\rightarrow\sigma_2$ is an isoclinism, then
\vspace{0.1cm}
\begin{enumerate}
\item $\gamma\sigma_1(x)=\sigma_2\beta^\prime(x)$, for all $x\in[M_1^*,L_1]$.
\vspace{0.1cm}
\item $\beta^\prime(\textnormal{Ker}\sigma_1\cap[M_1^*,L_1])=\textnormal{Ker}\sigma_2\cap[M_2^*,L_2]$.
\vspace{0.1cm}
\item $\beta^\prime(^{l_1}x)=~^{l_2}\beta^\prime(x)$, for all $x\in [M_1^*,L_1],l_i\in L_i,i=1,2$ with $\gamma\alpha_1(l_1)=\alpha_2(l_2)$.
\end{enumerate}
\end{lemma}

\begin{proof}
\begin{enumerate}
\item Let $x\in [M_1^*,L_1]$ such that $x=~^{l_1}\alpha_1^*(m_1)$, for some $l_1\in L_1$ and $m_1\in M_1^*$. Then, $$\gamma\sigma_1(x)=\gamma\sigma_1(^{l_1}\alpha_1^*(m_1))=\gamma[\alpha_1(l_1),\sigma_1(\alpha_1^*(m_1))]=[\gamma\alpha_1(l_1),\gamma\sigma_1(\alpha_1^*(m_1))].$$ Now if $\gamma\alpha_1(l_1)=\alpha_2(l_2)$ and $\gamma\sigma_1(\alpha_1^*(m_1))=\sigma_2(\alpha_2^*(m_2))$ where $l_2\in L_2$ and $m_2\in M_2^*$, then $$\gamma\sigma_1(x)=[\alpha_2(l_2),\sigma_2(\alpha_2^*(m_2))]=\sigma_2(^{l_2}\alpha_2^*(m_2))=\sigma_2(\beta^\prime(x)).$$
\item Let $x\in \textnormal{Ker}\sigma_1\cap[M_1^*,L_1]$, then $\sigma_1(x)=0$ and $x=~^{l_1}\alpha_1^*(m_1)$, for some $l_1\in L_1,~m_1\in M_1^*$. Suppose that $\gamma\alpha_1(l_1)=\alpha_2(l_2)$ and $\gamma\sigma_1(\alpha_1^*(m_1))=\sigma_2(\alpha_2^*(m_2))$, where $l_2\in L_2$ and $m_2\in M_2^*$. Now from the definition of isoclinism we have $\beta^\prime(^{l_1}\alpha_1^*(m_1))=~^{l_2}\alpha_2^*(m_2)$. \\Also we have, $$\sigma_2(^{l_2}\alpha_2^*(m_2))=[\alpha_2(l_2),\sigma_2(\alpha_2^*(m_2))]=\gamma[\alpha_1(l_1),\sigma_1(\alpha_1^*(m_1))]=\gamma\sigma_1(x)=0.$$ Therefore, $\beta^\prime(x)\in\textnormal{Ker}\sigma_2\cap[M_2^*,L_2]$ and the converse is obvious.
\item From (1), we have $\gamma\sigma_1(x)=\sigma_2(\beta^\prime(x))$. So by the isoclinic properties of the relative central extensions $\sigma_1$, $\sigma_2$ and for all $x\in [M_1^*,L_1], l_1\in L_1,\gamma\alpha_1(l_1)=\alpha_2(l_2)$, we can clearly have $\beta^\prime(^{l_1}x)=~^{l_2}\beta^\prime(x)$.
\end{enumerate}
\end{proof}

Let $(M_1,L_1)$ and $(M_2,L_2)$ be two pairs of Hom-Lie algebras and suppose that $\sigma_1:(M_1^*,\alpha_1^*)\rightarrow (L_1,\alpha_1)$ and $\sigma_2:(M_2^*,\alpha_2^*)\rightarrow (L_2,\alpha_2)$ are relative central extensions of these pairs, respectively. Let $\gamma:(L_1,\alpha_1)\rightarrow (L_2,\alpha_2)$ be an isomorphism such that $\gamma(M_1)=M_2$ and $$\overline{M^*}=\{(m_1,m_2)|m_i\in M_i^*,i=1,2,~\gamma\sigma_1(\alpha_1^*(m_1))=\sigma_2(\alpha_2^*(m_2))\}.$$ Then $\overline{M^*}$ is a subalgebra of $M_1^*\times M_2^*$. Also, the homomorphism
\begin{align*}
 \hat{\sigma}:\overline{M^*}&\longrightarrow L_1 \\(m_1,m_2)&\longmapsto \sigma_1(m_1)
\end{align*}
together with an action of $L_1$ on $\overline{M^*}$ defined by $^{l_1}(m_1,m_2)=(^{l_1}m_1,~^{l_2}m_2),$ where $\gamma\alpha_1(l_1)=\alpha_2(l_2)$ is a relative central extension of the pair $(M_1,L_1)$.

\begin{proposition}\label{pro1}
The isomorphism $\gamma:(L_1,\alpha_1)\rightarrow (L_2,\alpha_2)$ induces an isoclinism from $\sigma_1$ to $\sigma_2$ if and only if there exist isoclinic epimorphisms from $\hat{\sigma}$ onto $\sigma_1$ and $\sigma_2$.
\end{proposition}

\begin{proof}
The necessary condition holds true obviously. To prove the converse part, let $\gamma$ induces an isoclinism from $\sigma_1$ to $\sigma_2$. Then there exists an isomorphism $\beta^\prime: [M_1^*,L_1]\rightarrow [M_2^*,L_2]$ induced by $\gamma$ such that the pair $(\gamma,\beta^\prime)$ is an isoclinism from $\sigma_1$ to $\sigma_2$. Now according to the hypothesis and by Lemma \ref{lem1} (1) we have $$[\overline{M^*},L_1]=\{(m_1,\beta^\prime(m_1))|m_1\in [M_1^*,L_1]\}.$$ Let $\beta_i:\overline{M^*}\rightarrow M_i^*$ be the projective homomorphisms and let $\gamma_1=1_{L_1},\gamma_2=\gamma$. Now it can be easily checked that $(\beta_1,\gamma_1):\hat{\sigma}\rightarrow \sigma_1$ and $(\beta_2,\gamma_2):\hat{\sigma}\rightarrow\sigma_2$ are isoclinic epimorphisms and the result follows.
\end{proof}

Let $A=\overline{M^*}/[\overline{M^*},L_1]$. Then $A$ is an abelian Hom-Lie algebra. Thus by example 1, for any central relative extension $\sigma:M^*\rightarrow L$ of the pair $(M,L)$, $\sigma{\pi_{M^*}}:M^*\times A\rightarrow L$ is also a relative central extension. The following proposition explains more of this.

\begin{proposition}\label{pro2}
Let $(\gamma, \beta^\prime)$ be an isoclinism from $\sigma_1$ to $\sigma_2$. Then there exist isoclinic monomorphisms from $\hat{\sigma}$ into $\sigma_1\pi_{M_1^*}$ and $\sigma_2\pi_{M_2^*}$.
\end{proposition}

\begin{proof}
Let's consider a map $\overline{\beta_i}:\overline{M^*}\rightarrow M_i^*\oplus A$ given by $$\overline{\beta_i}(x)=(\beta_i(x),x+[\overline{M^*},L_1]),$$ for $i=1,2$ and $\beta_i$'s are as defined in Proposition \ref{pro1}. Clearly, $\overline{\beta_i}$ is a Hom-Lie homomorphism. Also we have, $[\overline{M^*},L_1]=\{(m_1,\beta^\prime(m_1))|m_1\in [M_1^*,L_1]\}.$ Now if $x=(m_1,m_2)\in \textnormal{Ker}\overline{\beta_i}$, then $m_i=0$ and $m_2=\beta^\prime(m_1)$, which implies that $\textnormal{Ker}\overline{\beta_i}=0$. Hence $\overline{\beta_i}$ is a Hom-Lie monomorphism and it can be easily seen that the morphisms $(\overline{\beta_1},1_{L_1}):\hat{\sigma}\rightarrow\sigma_1{\pi_{M_1^*}}$ and $(\overline{\beta_2},\gamma):\hat{\sigma}\rightarrow\sigma_2{\pi_{M_2^*}}$ are isoclinic monomorphisms.
\end{proof}

\begin{proposition}\label{pro3}
The isomorphism $\gamma:(L_1,\alpha_1)\rightarrow (L_2,\alpha_2)$ induces an isoclinism from $\sigma_1$ to $\sigma_2$ if and only if for some ideal $T$ of $M_1^*\oplus A$ with $T\subseteq \textnormal{Ker}\sigma_1{\pi_{M_1^*}}$, there exist isoclinic monomorphisms from $\sigma_1$ and $\sigma_2$ into the relative central extension $\overline{\sigma_1{\pi_{M_1^*}}}:(M_1^*\times A)/T\rightarrow L_1$.
\end{proposition}

\begin{proof}
The necessary condition is obvious. To prove the converse part, let's suppose that $\gamma $ induces an isoclinism $(\gamma, \beta^\prime)$ from $\sigma_1$ to $\sigma_2$. Let $T=\{(x,(x,0)+[\overline{M^*},L_1])|x\in\textnormal{Ker}\sigma_1\}$. Then clearly $T$ is an ideal of $M_1^*\times A$ with $T\subseteq \textnormal{Ker}\sigma_1{\pi_{M_1^*}}$. \\
\vspace{0.00001cm}
Now, we define two Hom-Lie homomorphisms $\delta_1:M_1^*\rightarrow (M_1^*\times A)/T$ and $\delta_2: M_2^*\rightarrow (M_1^*\times A)/T$ as follows: $$\delta_1(m_1)=(m_1,[\overline{M^*},L_1])+T$$ and $$\delta_2(m_2)=(m_1,(m_1,m_2)+[\overline{M^*},L_1])+T,$$ where $m_i\in M_i^*~(i=1,2)$ and $\gamma\sigma_1(\alpha_1^*(m_1))=\sigma_2(\alpha_2^*(m_2))$. We claim that $\delta_1$ and $\delta_2$ are monomorphisms. If $m_1\in \textnormal{Ker}\delta_1$, then $m_1\in \textnormal{Ker}\sigma_1$ and $\beta^\prime(m_1)=0$ and thereby $m_1=0$. Similarly, $\textnormal{Ker}\delta_2=0$. Now it can be easily checked that the morphisms $(\delta_1,1_{L_1}):\sigma_1\rightarrow\overline{\sigma_1{\pi_{M_1^*}}}$ and $(\delta_2, \gamma^{-1}):\sigma_2\rightarrow\overline{\sigma_1{\pi_{M_1^*}}}$ are isoclinic monomorphisms, as required.
\end{proof}

\section{\textbf{Factor set on an isoclinic relative extension}}

%\textbf{Definition.} Let $(M,L)$ be a pair of Lie superalgebras. Let $\bar{M}=M/Z(M,L)$ and $\bar{M}=\bar{M}_0\oplus \bar{M}_1$. A \textit{factor set} on the pair $(M,L)$ is defined as a bilinear map $$f:M/Z(M,L)\times M/Z(M,L)\longrightarrow Z(M,L)$$ such that the following conditions are satisfied:
%\begin{enumerate}
%\item $f(\bar{M}_\alpha,\bar{M}_\beta)\subseteq Z(M,L)_{\alpha+\beta},~\alpha,\beta\in\mathbb{Z}_2$
%\item $f(\bar{m_1},\bar{m_2})=-(-1)^{|\bar{m_1}||\bar{m_2}|}f(\bar{m_2},\bar{m_1})$
%\item $f([\bar{m_1},\bar{m_2}],\bar{m_3})=f(\bar{m_1},[\bar{m_2},\bar{m_3}])-(-1)^{{|\bar{m_1}||\bar{m_2}|}}f(\bar{m_2},[\bar{m_1},\bar{m_3}])$
%\end{enumerate}
%for all $\bar{m_1},\bar{m_2},\bar{m_3}\in \bar{M}$.

\begin{definition}
Let $(M,L)$ be a pair of Hom-Lie algebras. A relative central extension $\sigma:(M^*,\alpha_*)\rightarrow (L,\alpha_L)$ is said to be \textit{finite dimensional} , if $(M^*,\alpha_*)$ is a finite dimensional Hom-Lie algebra.
\end{definition}

\begin{definition}
Let $\sigma:(M^*,\alpha_*)\rightarrow (L,\alpha_L)$ be a relative central extension of a pair of Hom-Lie algebras $(M,L)$. A \textit{factor set} on $\sigma$ is defined as a skew-bilinear map $$f:(L,\alpha_L)\times (L,\alpha_L)\rightarrow\textnormal{Ker}(\sigma)$$ such that the following conditions are satisfied:
\begin{enumerate}
\item $f(l_1,l_1)=0$
\vspace{0.1cm} 
\item $f([{l_1},{l_2}],\alpha_L({l_3}))+f([{l_2},{l_3}],\alpha_L({l_1}))+f([{l_3},{l_1}],\alpha_L({l_2}))=0$
\end{enumerate}
\vspace{0.1cm} 
for all ${l_1},{l_2},{l_3}\in L$.
\end{definition}

\vspace{0.3cm} 
\begin{lemma}\label{lem101}
Let $f$ be a factor set on the relative central extension $\sigma$ of the pair $(M,L)$. Then the product $\textnormal{Ker}\sigma\times M$ is a Hom-Lie algebra with the bracket defined as, $$[(x_1,m_1),(x_2,m_2)]=(f(m_1,m_2),[m_1,m_2]),$$ along with a linear map ${\alpha_f}:\textnormal{Ker}\sigma\times M\rightarrow\textnormal{Ker}\sigma\times M$ such that ${\alpha_f}(x,m)=(\alpha_*(x),\alpha_L(m)),$ where $\alpha_*$ and $\alpha_L$ are as defined in Definition 3.2. 
\end{lemma}
\begin{proof}
First we will check the anti-symmetry
\begin{align*}
[(x_1,m_1),(x_2,m_2)]&=(f(m_1,m_2),[m_1,m_2]) \\&=(-f(m_2,m_1),-[m_2,m_1]) \\&=-[(x_2,m_2),(x_1,m_1)].
\end{align*}
Now to check the Hom-Jacobi identity, we have
\begin{align*}
[\alpha_f(x_1,m_1),[(x_2,m_2),(x_3,m_3)]] &=[(\alpha_*(x_1),\alpha_L(m_1)),(f(m_2,m_3),[m_2,m_3])] \\&=(f(\alpha_L(m_1),[m_2,m_3]),[\alpha_L(m_1),[m_2,m_3]]),
\end{align*}
by performing similar calculations we can find that 
\begin{align*}
	[\alpha_f(x_1,m_1),[(x_2,m_2),(x_3,m_3)]] &+[\alpha_f(x_2,m_2),[(x_3,m_3),(x_1,m_1)]] \\&+[\alpha_f(x_3,m_3),[(x_1,m_1),(x_2,m_2)]]=0.
\end{align*}
Therefore $\textnormal{Ker}\sigma\times M$ is a Hom-Lie algebra.
\end{proof}
We denote the above Hom-Lie algebra by $(\textnormal{Ker}\sigma\times M)_f$. Then the projective map, 
\begin{align*}
\sigma_f:(\textnormal{Ker}\sigma\times M)_f&~\rightarrow ~L \\(x,m)&~\mapsto~ m
\end{align*}
together with an action of $L$ on $(\textnormal{Ker}\sigma\times M)_f$ defined by $^l(x,m)=(f(m,l),[m,l])$is a relative central extension of the pair $(M,L)$.

\vspace{0.3cm} 

\begin{definition}
Let $\sigma:(M^*,\alpha_*)\rightarrow (L,\alpha_L)$ be a relative central extension of a pair of Hom-Lie algebras $(M,L)$.  If $\textnormal{Ker}\sigma\subseteq [M^*,L]$, then $\sigma$ is called as a \textit{stem relative central extension}.
\end{definition}

\begin{lemma}\label{lem11}
Let $\sigma_1:(M_1^*,\alpha_1^*)\rightarrow (L_1,\alpha_1)$ and $\sigma_2:(M_2^*,\alpha_2^*)\rightarrow (L_2,\alpha_2)$ be two relative central extensions of the pairs $(M_1,L_1)$ and $(M_2,L_2)$ of finite dimensional Hom-Lie algebras, respectively. Then, 
\begin{enumerate}
\item $\sigma_1$ is a stem relative central extension if and only if $\textnormal{Ker}\sigma_1$ contains no non-zero ideal $N$ satisfying $N\cap[M_1^*,L_1]=0$.
\item If $\sigma_1$ and $\sigma_2$ are isoclinic stem relative central extensions, then $\textnormal{Ker}\sigma_1\cong\textnormal{Ker}\sigma_2$. In particular, if the Hom-Lie algebras $M_1^*$ and $M_2^*$ are finite dimensional, then $\textnormal{dim}M_1^*=\textnormal{dim}M_2^*$.
\end{enumerate}
\end{lemma}

\begin{proof}
\begin{enumerate}
\item Let's suppose that $\sigma_1$ is stem and $N$ is an ideal of $M_1^*$ with $N\subseteq \textnormal{Ker}\sigma_1$ and $ N\cap[M_1^*,L_1]=0$. But we know that $\textnormal{Ker}\sigma_1\subseteq Z(M_1^*,L_1)$. Also, since $\sigma_1$ is stem, so $\textnormal{Ker}\sigma_1\subseteq[M_1^*,L_1]$. Therefore, we have $$N\subseteq N\cap\textnormal{Ker}\sigma_1\subseteq N\cap Z(M_1^*,L_1)\cap [M_1^*,L_1]=N\cap [M_1^*,L_1]=0,$$ and the result follows.

\vspace{0.1cm} 

Conversely, let $ \textnormal{Ker}\sigma_1\subsetneq Z(M_1^*,L_1)\cap [M_1^*,L_1].$ Then there exists a non-zero element $x\in\textnormal{Ker}\sigma_1\setminus (Z(M_1^*,L_1)\cap [M_1^*,L_1]) $. Now $\langle x\rangle$ is clearly a central ideal of $M_1^*$ and consequently, by hypothesis, we have $\langle x\rangle\cap [M_1^*,L_1]\neq 0$. So, for some non-zero element $a\in F, ~ ax\in [M_1^*,L_1]$, which is a contradiction. Therefore, $\textnormal{Ker}\sigma_1\subseteq Z(M_1^*,L_1)\cap [M_1^*,L_1]$ and $\sigma_1$ is a stem relative central extension.

\item Let $\sigma_1$ and $\sigma_2$ be isoclinic stem relative central extensions, then there exists an isomorphism $\gamma :(L_1,\alpha_1)\rightarrow (L_2,\alpha_2)$ with $\gamma(M_1)=M_2$. Therefore, $\textnormal{dim}M_1=\textnormal{dim}M_2$. Also, by Lemma \ref{lem1} (2), $\textnormal{Ker}\sigma_1\cong \textnormal{Ker}\sigma_2$. If $M_1^*$ and $M_2^*$ are finite dimensional, then by the isomorphism theorems, we have $\textnormal{dim}M_1^*=\textnormal{dim}M_2^*$.
\end{enumerate}
\end{proof}

The next result shows that every isoclinic family of the relative central extensions of pairs of Hom-Lie algebras contains a stem relative central extension.

\begin{corollary}\label{cor1}
 Any relative central extension is isoclinic to a stem relative central extension.
\end{corollary}

\begin{proof}
Let $\mathcal{A}=\{N|N \textnormal{~is~an~ideal~of~Ker}\sigma ~\textnormal{and}~N\cap[M^*,L]=0\}$, where $\sigma:M^*\rightarrow L$ is a relative central extension of $(M,L)$. Then the set $\mathcal{A}$ is non-empty, as it contains the zero ideal. Now if we define a partial ordering on $\mathcal{A}$ by inclusion then by Zorn's lemma, we can find a maximal ideal $N$ in $\mathcal{A}$. Since, $N\subset\textnormal{Ker}\sigma$, by example 2, $\overline{\sigma}:M^*/n\rightarrow L$ is a relative central extension and $\sigma\sim\overline{\sigma}$.

\vspace{0.1cm} 

Now, let $K/N$ is an ideal of $M^*/N$ such that $K/N\subseteq\textnormal{Ker}\overline{\sigma}$ and $K/N\cap[M^*/N,L]=0$. Then $(K\cap[M^*,L]+N)/N=0$ and hence $K\cap[M^*,L]\subseteq N$. Since $N\cap[M^*,L]=0$, we conclude that $K\in \mathcal{A}$. Also, we have $N\subseteq K$. So by the maximality of $N$, it follows that $K=N$. Therefore, $K/N$ is zero ideal and hence $\overline{\sigma}$ is a stem relative central extension by Lemma \ref{lem11} (1).
\end{proof}

\begin{lemma} \label{lem2}
Let $\sigma_1:(M_1^*,\alpha_1^*)\rightarrow (L_1,\alpha_1)$ and $\sigma_2:(M_2^*,\alpha_2^*)\rightarrow (L_2,\alpha_2)$ be relative central extensions of the pairs of Hom-Lie algebras $(M_1,L_1)$ and $(M_2,L_2)$, respectively. Then
\begin{enumerate}
\item There exists a factor set $f:(L_1,\alpha_1)\times (L_1,\alpha_1)\rightarrow\textnormal{Ker}(\sigma_1)$ such that the relative central extensions $\sigma_1$ and $\sigma_{1f}$ are isomorphic.
\item If $\sigma_1$ and $\sigma_{2}$ are stem and $(\gamma,\beta^\prime):\sigma_1\rightarrow\sigma_{2}$ is an isoclinism, then there exists a factor set $g:(L_1,\alpha_1)\times (L_1,\alpha_1)\rightarrow\textnormal{Ker}(\sigma_1)$ such that the relative central extensions $\sigma_1$ and $\sigma_{1g}$ are isomorphic.
\end{enumerate}
\end{lemma}

\begin{proof}
\begin{enumerate}
\item Let $T_1$ be the vector space complement of $\textnormal{Ker}(\sigma_1)$ in $M_1^*$. Then for any $m_1\in M_1$, there is a unique element $t_{m_1}\in T_1$ with $\sigma_1(t_{m_1})=m_1$. 
Therefore we have 
\begin{align*}
\sigma_1(t_{[m_1,m_2]})=[m_1,m_2]=-[m_2,m_1]&=-\sigma_1(t_{[m_2,m_1]})\\&=\sigma_1(-t_{[m_2,m_1]}), 
\end{align*}
which implies that $t_{[m_1,m_2]}=-t_{[m_2,m_1]}$ for all $m_1,m_2\in M_1.$

\vspace{0.1cm} 

%Also by taking 
%\begin{align*}
%\sigma_1(t_{[[m_1,m_2],m_3]})=[[m_1,m_2],m_3]&=[m_1,[m_2,m_3]]-(-1)^{|m_1||m_2|}[m_2,[m_1,m_3]]\\&=\sigma_1(t_{[m_1,[m_2,m_3]]})-(-1)^{|m_1||m_2|}\sigma_1(t_{[m_2,[m_1,m_3]]}),
%\end{align*}

%we have, $t_{[[m_1,m_2],m_3]}=t_{[m_1,[m_2,m_3]]}-(-1)^{|m_1||m_2|}t_{[m_2,[m_1,m_3]]}$ for all $m_1,m_2,m_3\in M_1.$

%\vspace{0.3cm} 

Again since $\sigma_1$ is a homomorphism, so we have 
\begin{align*}
\sigma_1([t_{[m_1,m_2]},t_{\alpha_1(m_3)}])=[\sigma_1({t_{[m_1,m_2]}}),\sigma_1({t_{\alpha_1(m_3)}})]=[[m_1,m_2],\alpha_1(m_3)]=\sigma_1(t_{[[m_1,m_2],\alpha_1(m_3)]})
\end{align*}
$$\implies[t_{[m_1,m_2]},t_{\alpha_1(m_3)}]=t_{[[m_1,m_2],\alpha_1(m_3)]}.$$

\vspace{0.1cm}

Let's define a map $f:(L_1,\alpha_1)\times (L_1,\alpha_1)\rightarrow\textnormal{Ker}(\sigma_1)$ given by $f(m_1,m_2)=[t_{m_1},t_{m_2}]-t_{[m_1,m_2]}$ for all $m_1,m_2\in M_1$ and $f(x,y)=0$, if $x,y\in L_1\setminus M_1$. We will show that $f$ is a factor set.

\vspace{0.1cm} 

From the definition it is clear that $f$ satisfies all the conditions of a factor set on $L_1\setminus M_1$. So it only remains to show that $f$ is also a factor set on $M_1$. It is obvious that $f(m_1,m_2)\in\textnormal{Ker}(\sigma_1)$. Now to check the skew-symmetric property of $f$, we take for all  $m_1,m_2\in M_1$
\begin{align*}
f(m_1,m_2)&=[t_{m_1},t_{m_2}]-t_{[m_1,m_2]}\\&=-[t_{m_2},t_{m_1}]+t_{[m_2,m_1]}\\&=-\{[t_{m_2},t_{m_1}]-t_{[m_2,m_1]}\}\\&=-f(m_2,m_1).
\end{align*}

To check the Jacobi identity, we take for all $m_1,m_2,m_3\in M_1$
\begin{align*}
f([m_1,m_2],\alpha_1(m_3))=[t_{[m_1m_2]},t_{\alpha_1(m_3)}]-t_{[[m_1,m_2],\alpha_1(m_3)]}=0,
\end{align*}
and similarly, $f([{m_2},{m_3}],\alpha_1({m_1}))=0=f([{m_3},{m_1}],\alpha_1({m_2}))$.
\\i.e., $f$ also satisfies the Jacobi identity and hence is a factor set. Now consider the mapping $\beta:(\textnormal{Ker}\sigma_1\times M_1)_f\rightarrow M_1^*$ given by $\beta(x_1,m_1)=x_1+t_{m_1}$. Then we have, $(x_1,m_1)=(x_2,m_2)\Leftrightarrow x_1+t_{m_1}=x_2+t_{m_2}\Leftrightarrow \beta{(x_1,m_1)}=\beta{(x_2,m_2)}$ i.e., $\beta$ is well defined and obviously bijective. Also, 
\begin{align*}
\beta([(x_1,m_1),(x_2,m_2)])=\beta(f(m_1,m_2),[m_1,m_2])&=f(m_1,m_2)+t_{[m_1,m_2]}\\&=[t_{m_1},t_{m_2}]\\&=[x_1+t_{m_1},x_2+t_{m_2}]\\&=[\beta{(x_1,m_1)},\beta{(x_2,m_2)}].
\end{align*}
i.e., $\beta$ preserves the bracket operation. Again, we know that the relative central extension $\sigma_1:(M_1^*,\alpha_1^*)\rightarrow (L_1,\alpha_1)$ being a Hom-Lie homomorphism satisfies $\sigma_1\alpha_1^*=\alpha_1\sigma_1$. So, we have $\sigma_1\alpha_1^*(t_m)=\alpha_1\sigma_1(t_m)=\alpha_1(m)$, which implies that $\alpha_1^*(t_m)=t_{\alpha_1(m)}$. Now for some $(x,m)\in (\textnormal{Ker}\sigma_1\times M_1)_f$, we get 
\begin{align*}
\beta{\alpha_f}(x,m)=\beta(\alpha_1^*(x),\alpha_1(m))&=\alpha_1^*(x)+t_{\alpha_1(m)}\\&=\alpha_1^*(x)+\alpha_1^*(t_m)\\&=\alpha_1^*(x+t_m)\\&=\alpha_1^*\beta(x,m),
\end{align*}
i.e., $\beta{\alpha_f}=\alpha_1^*\beta$, which confirms that $\beta$ is a Hom-Lie isomorphism. Therefore, $(1_{L_1},\beta):\sigma_{1f}\rightarrow \sigma_1$ is an isomorphism between the central extensions $\sigma_{1f}$ and $\sigma_1$.
\vspace{0.1cm} 
\item By Lemma \ref{lem1}(2), we have $\beta^\prime(\textnormal{Ker}\sigma_1)=\textnormal{Ker}\sigma_2$. Since $\sigma_2$ is a stem relative extension, so from the above proof, we can conclude that there exists a factor set $h:(L_2,\alpha_2)\times (L_2,\alpha_2)\rightarrow\textnormal{Ker}(\sigma_2)$ such that the relative central extensions $\sigma_2$ and $\sigma_{2h}$ are isomorphic. Define a map $g:(L_1,\alpha_1)\times (L_1,\alpha_1)\rightarrow\textnormal{Ker}(\sigma_1)$ given by $g(m_1,m_2)=\beta^{\prime-1}(h(\gamma(m_1),\gamma(m_2)))$ for all $m_1,m_2\in M_1$ and otherwise $g(x,y)=0.$ It can be easily shown that $g$ is a factor set.
\vspace{0.1cm} \\
Now consider a map $\theta:(\textnormal{Ker}\sigma_1\times M_1)_g\rightarrow (\textnormal{Ker}\sigma_2\times M_2)_h$ defined by $\theta(x_1,m_1)=(\beta^\prime(x_1),\gamma(m_1))$. It is obvious that $\theta$ is well defined and bijective. Also $\theta$ is a Hom-Lie homomorphism as,
\begin{align*}
\theta[(x_1,m_1),(x_2,m_2)]_g&=\theta(g(m_1,m_2),[m_1,m_2])\\&=\theta(\beta^{\prime-1}(h(\gamma(m_1),\gamma(m_2))),[m_1,m_2])\\&=(\beta^{\prime}(\beta^{\prime-1}(h(\gamma(m_1),\gamma(m_2)))),\gamma([m_1,m_2]))\\&=(h(\gamma(m_1),\gamma(m_2)),[\gamma(m_1),\gamma(m_2)])\\&=[(\beta^\prime(x_1),\gamma(m_1)),(\beta^\prime(x_2),\gamma(m_2))]_h\\&=[\theta(x_1,m_1),\theta(x_2,m_2)]_h,
\end{align*}
and if $\alpha_g$ and $\alpha_h$ are the linear maps in the Hom-Lie algebras $(\textnormal{Ker}\sigma_1\times M_1)_g$ and $(\textnormal{Ker}\sigma_2\times M_2)_h$ respectively, then
\begin{align*}
\theta\alpha_g(x_1,m_1)=\theta(\alpha_1^*(x_1),\alpha_1(m_1))&=(\beta^\prime\alpha_1^*(x_1),\gamma\alpha_1(m_1))\\&=(\alpha_2^*\beta^\prime(x_1),\alpha_2\gamma(m_1))\\&=\alpha_h(\beta^\prime(x_1),\gamma(m_1))\\&=\alpha_h\theta(x_1,m_1).
\end{align*}
i.e., $\theta$ is a Hom-Lie isomorphism and hence $(\gamma,\theta)$ is an isomorphism between $\sigma_{1g}$ and $\sigma_{2h}$. Therefore, $\sigma_{1g}\cong\sigma_{2h}\cong\sigma_{2}$.
\end{enumerate}
\end{proof}

\begin{proposition}\label{pro4}
Let $\sigma:(M^*,\alpha_*)\rightarrow (L,\alpha_L)$ be a finite dimensional relative central extension and $f,g: (L,\alpha_L)\times  (L,\alpha_L)\rightarrow\textnormal{Ker}\sigma$ be factor sets on $\sigma$ such that $\textnormal{dim}(\textnormal{Ker}\sigma\times M)_f=\textnormal{dim}(\textnormal{Ker}\sigma\times M)_g$, $\sigma_f$ is a stem relative extension and the extensions  $\sigma_f$ and $\sigma_g$ are isoclinic. Then $\sigma_f$ and $\sigma_g$ are isomorphic.
\end{proposition}

\begin{proof}
Let $(\gamma,\beta^\prime)$ be an isoclinism between $\sigma_f$ and $\sigma_g$. Then from Lemma \ref{lem1}(2), it follows that $\beta^\prime (\textnormal{Ker}\sigma_f)=\textnormal{Ker}\sigma_g$. Now for all $m\in M,l\in L$, we have $$\beta^\prime(^l(0,m))=~^{\gamma(l)}(0,\gamma(m))=(g(\gamma(m),\gamma(l)),\gamma([m,l])),$$ and $$\beta^\prime(^l(0,m))=\beta^\prime(f(m,l),[m,l])=\beta^\prime(f(m,l),0)+\beta^\prime(0,[m,l]).$$ 

\vspace{0.1cm} 

Let $d([m,l])$ be the first component of $\beta^\prime(^l(0,m))$. Then we can define a linear map $d:[M,L]\rightarrow\textnormal{Ker}\sigma$ satisfying $\rho\beta^\prime(f(m,l),0)+d([m,l])=g(\gamma(m),\gamma(l))$, where $\rho:\textnormal{Ker}\sigma_g\rightarrow\textnormal{Ker}\sigma$ is the projective map. If we take $d$ to be zero on the vector space complement of $[M,L]$ in $M$ then $d$ can be defined on whole $M$. Now it is easy to see that the map $\lambda:(\textnormal{Ker}\sigma\times N)_f\rightarrow(\textnormal{Ker}\sigma\times N)_g$ given by $\lambda(x,l)=\beta^\prime(x,0)+(d(l),\gamma(l))$, is an isomorphism. As $\lambda|_{\textnormal{Ker}\sigma_f}=\beta^\prime|_{\textnormal{Ker}\sigma_f}$ and $(\lambda,\gamma):\sigma_f\rightarrow\sigma_g$ is a morphism, we conclude that $\sigma_f$ and $\sigma_g$ are isomorphic.
\end{proof}

\begin{corollary}\label{cor2}
Let $\sigma_1:(M_1^*,\alpha_1^*)\rightarrow (L_1,\alpha_1)$ and $\sigma_2:(M_2^*,\alpha_2^*)\rightarrow (L_2,\alpha_2)$ be finite dimensional stem relative central extensions of the pairs $(M_1,L_1)$ and $(M_2,L_2)$ respectively. Then $\sigma_1$ and $\sigma_2$ are isoclinic if and only if they are isomorphic.
\end{corollary}

\begin{proof}
Let $\sigma_1$ and $\sigma_2$ are isoclinic. Then by lemma \ref{lem2}, there exist factor sets $f,g:(L_1,\alpha_1)\times  (L_1,\alpha_1)\rightarrow\textnormal{Ker}\sigma_1$ such that $\sigma_1\cong \sigma_{1f}$ and $\sigma_2\cong \sigma_{1g}$. Therefore, $ \sigma_{1f}$ and $ \sigma_{1g}$ are isoclinic and since $(\textnormal{Ker}\sigma_1\times M_1)_f$ and  $(\textnormal{Ker}\sigma_1\times M_1)_g$ are finite dimensional, so we can have $\textnormal{dim}(\textnormal{Ker}\sigma_1\times M_1)_f=\textnormal{dim}(\textnormal{Ker}\sigma_1\times M_1)_g$. Hence by Proposition \ref{pro4}, $\sigma_{1f}\cong\sigma_{1g}.$ Therefore, $\sigma_1\cong\sigma_{1f}\cong\sigma_{1g}\cong\sigma_2$ and the result follows.
\end{proof}

\vspace{0.1cm} 

\noindent{\sc Example 3.} Let $\sigma:(M^*,\alpha_*)\rightarrow (L,\alpha_L)$ be a relative central extension of $(M,L)$. Let $T$ be a Hom-Lie subalgebra of $M^*$ such that $M^*=T+\textnormal{Ker}\sigma$ and also $\sigma(T)=M$. Then the map $\sigma_T:T\rightarrow L$ defined by $\sigma_T(t)=\sigma(t)$ for all $t\in T$ is a relative central extension of $(M,L)$ and $(1_L,i):\sigma_T\rightarrow \sigma$ is an isoclinic monomorphism, where $i:T\rightarrow M^*$ is an inclusion map.

\vspace{0.1cm}

The next theorem helps to determine the structure of all the relative central extensions in an isoclinism family of a given relative central extension.

\begin{theorem}\label{thm1}
Let $\mathcal{C}$ be an isoclinism family of finite dimensional relative central extensions. Then $\mathcal{C}$ contains a stem relative central extension $\sigma_1:(M_1^*,\alpha_1^*)\rightarrow (L_1,\alpha_1)$ and relative central extension in $\mathcal{C}$ is the form of $\sigma_1\pi_{M_1^*}:M_1^*\times A\rightarrow L_1$, in which $A$ is a finite dimensional abelian Hom-Lie algebra.
\end{theorem}

\begin{proof}
From Corollary \ref{cor1}, it is clear that $\mathcal{C}$ contains a stem relative central extension $\sigma_1$. We know that for any finite dimensional abelian Hom-Lie algebra $A$, the relative central extensions $\sigma_1:M_1^*\rightarrow L_1$ and $\sigma_1\pi_{M_1^*}:M_1^*\times A\rightarrow L_1$ are isoclinic and consequently $\sigma_1\pi_{M_1^*}\in \mathcal{C}$. 

\vspace{0.1cm} 

Now, let $\sigma:M^*\rightarrow L$ be an arbitrary relative central extension in $\mathcal{C}$. Then there exists an ideal $N$ of $\textnormal{Ker}\sigma$ such that $\textnormal{Ker}\sigma=N\times(\textnormal{Ker}\sigma\cap(M^*)^2)$, where $(M^*)^2$ is the derived Hom-Lie subalgebra of $M^*$ and by the proof of Corollary \ref{cor1}, and example 2, the map $\bar{\sigma}:M^*/N\rightarrow L$ is a stem relative central extension in $\mathcal{C}$. If we choose the subspace $T$ of $M^*$ such that $(M^*)^2\subseteq T$ and $T\times N=M^*$, then $T$ is a Hom-Lie subalgebra of $M^*$ with $\sigma(T)=M$ and, $\sigma_1$ and $\bar{\sigma}$ are isoclinic. By Corollary \ref{cor2}, we get $\sigma_1\cong\sigma_T$. Therefore, $\sigma\cong\sigma_T\pi_T\cong\sigma_1\pi{M_1^*}$, where $\sigma_T\pi_T:T\times N\rightarrow L$ and $\sigma_1\pi{M_1^*}:M_1^*\times N\rightarrow L$ are relative central extensions of the pair $(M,L)$ and $N$ is a finite dimensional Hom-Lie algebra.
\end{proof}

\end{document}